\documentclass[11pt, english, a4paper]{amsart}
\usepackage[margin=3.2cm]{geometry}
\usepackage{amssymb,amsmath,amsthm}
\usepackage{mathtools}
\usepackage{mathrsfs}
\usepackage{accents}
\usepackage[utf8]{inputenc}
\usepackage[T1]{fontenc}

\usepackage{caption}
\usepackage{subcaption}

\usepackage{xfrac}
\usepackage{enumerate}
\usepackage{graphics, float}

\usepackage[colorlinks]{hyperref}
\hypersetup{allcolors = black}

\usepackage{url}

\usepackage{svg}

\usepackage{comment}

\newtheorem{theorem}{Theorem}
\newtheorem{proposition}{Proposition}

\newcommand{\R}{\mathbb{R}}
\newcommand{\la}{\langle}
\newcommand{\ra}{\rangle}
\newcommand{\I}{\mathcal{I}}

\newcommand{\D}{\mathrm{D}}
\newcommand{\E}{\mathcal{E}}

\newcommand{\LL}{\mathcal{L}}
\newcommand{\HH}{\mathcal{H}}

\title{An elementary solution of the polarization problem}
\author{Gergely Ambrus}
\date{\today}

\begin{document}

\begin{abstract}
By distilling and combining earlier arguments, we give a self-contained
and entirely elementary proof of the strong and real linear polarization
inequalities established by Martínez and Ortega-Moreno. Specifically,
for any unit vectors $u_1,\ldots,u_n\in\R^d$, there exists a unit vector
$v\in\R^d$ such that
\[
\sum_{i=1}^n \frac{1}{\langle u_i,v\rangle^2}\le n^2.
\]
As an immediate consequence, we recover the real linear polarization
inequality
\[
\prod_{i=1}^n |\langle u_i,v\rangle|\ge n^{-n/2}.
\]
\end{abstract}

\maketitle

\section{Introduction}

The study of linear polarization constants was initiated in the late 1990s by Ryan and Turett~\cite{RyanTurett1998} and Benítez, Sarantopoulos and Tonge~\cite{BenitezSarantopoulosTonge1998}; the latter work gave rise to the real linear polarization conjecture. The stronger conjecture stated below was subsequently proposed by Ball and Frenkel; see~\cite[Conjecture~1.3]{Ambrus2009}.

Recently, both the real linear and the strong polarization problems have been resolved and subsequently generalized in a series of papers. Martínez and Ortega-Moreno
\cite{MartinezOrtegaMoreno2026} gave the first proof of the strong
polarization inequality, based on the Euler--Jacobi vanishing theorem; the
real linear polarization conjecture follows as a consequence. (The planar case was proved  earlier by the present author, Ball and Erdélyi~\cite{Ambrus2009,AmbrusBallErdelyi2013}). Shortly afterwards, Galicer, Ortega-Moreno and Pinasco
\cite{GalicerOrtegaMorenoPinasco2026} established a weighted version of the strong polarization inequality, with further consequences for products of powers of linear functionals and for plank problems. Ouimet and Greaves \cite{OuimetGreaves2026} then proved the Strong Gaussian Product Inequality, which in particular yields the real linear polarization inequality and its
weighted multiplicative extension. Martínez and Ortega-Moreno \cite{MartinezOrtegaMoreno2026_Aomoto} subsequently gave an alternative proof, as well as a generalization of the Gaussian Product Inequality, and characterized the extremal configurations for the strong polarization
inequality. Their general Aomoto interpolation theorem relies on the dimension theorem of Orlik and Terao~\cite{OrlikTerao94}; in the basis case, the relevant dimension is simply $2^n$. Finally, the present author \cite{Ambrus26} proved an Inverse Spectral Theorem which places these identities in a more general inverse-eigenvector framework and again yields the strong polarization inequality, together with weighted, matrix-valued, and other generalizations. We note that, in the positive definite setting, the inverse spectral decomposition is already contained in \cite{OuimetGreaves2026,MartinezOrtegaMoreno2026_Aomoto}. Artificial intelligence tools also played a role, to varying degrees, in several of these works.

The goal of the present note is to give a self-contained and  entirely elementary proof of the polarization problems by distilling and combining the ideas from the above articles, while emphasizing the geometry of the problem. The main objective is the following result.

\begin{theorem}[The strong polarization inequality, \cite{MartinezOrtegaMoreno2026}] \label{thm:strongpol}
Let $u_1, \ldots, u_n \in \R^d$ be unit vectors. Then there exists a unit vector $v \in \R^d$ that satisfies
\begin{equation*}\label{eq:strongpol}
\sum_{i=1}^n \frac {1}{\la u_i, v \ra^2} \leq n^2.
\end{equation*}
\end{theorem}

By the arithmetic-geometric mean inequality, this readily implies the product inequality:

\begin{theorem}[The $n$th real linear polarization inequality, \cite{MartinezOrtegaMoreno2026}] \label{thm:pol}
Let $u_1, \ldots, u_n \in \R^d$ be unit vectors. Then there exists a unit vector $v \in \R^d$ that satisfies
\begin{equation*}\label{eq:prodpol}
\prod_{i=1}^n |\la u_i, v \ra| \geq n^{-n/2}.
\end{equation*}
\end{theorem}

The proof  closely follows the basis case of the Aomoto-space interpolation argument developed by Martínez and Ortega-Moreno~\cite[Section 2]{MartinezOrtegaMoreno2026_Aomoto}. Under reciprocal coordinates, their interpolation formula is precisely the multi-affine formula used below. 
The positive definite algebraic ingredients are already contained in Ouimet–Greaves \cite{OuimetGreaves2026}: this interpolation formula follows directly from their determinant orthogonality argument, while the inverse spectral decomposition is explicitly contained in their equation (7), specialized to $\mathbf{y}=0$ and $c_i=1$. The coefficient-space formulation in the present work emphasizes the ellipsoid interpretation of inverse eigenvectors and makes the passage to singular correlation matrices immediate by compactness.

\section{The proof}

Following the approach of~\cite{Ambrus2009, Ambrus26}, we  work in the coefficient space, and search $w = \sqrt{n} \, v$ in the form 
\[
w = \sum_{i} \alpha_i u_i
\]
with $\alpha = (\alpha_1, \ldots, \alpha_n) \in \R^n$.
Then the condition $|v|=1$ is equivalent to 
\[
\alpha^\top M \alpha =n
\]
where $M = (m_{ij})_{i,j=1}^n$ is the Gram matrix of the vector system $\{ u_i \}_{i=1}^n$ defined by
\[
m_{ij} = \la u_i, u_j \ra.
\]
Note that $M$ is a {\em correlation matrix}, that is, a real positive semidefinite matrix whose diagonal entries are 1. Accordingly, set
\[
\E_M=\{x\in\R^n:x^\top Mx=n\}.
\]
Since $\sqrt{n}\la u_i, v \ra = \la u_i, w \ra = (M \alpha)_i$, Theorem~\ref{thm:strongpol} follows from the existence of $\alpha \in \E_M$ with 
\begin{equation}\label{eq:Malpha}
    \sum_i \frac 1 {(M \alpha)_i^2} \leq n.
\end{equation}
For vectors $\alpha$ that satisfy a certain algebraic relation, the inequality above takes a very appealing form.

For $\alpha=(\alpha_1, \ldots, \alpha_n)\in \R^n$ with nonzero coordinates, we write $\alpha^{-1}$ for its coordinatewise reciprocal:
\[
\alpha^{-1}
=
\left(\frac1{\alpha_1},\dots,\frac1{\alpha_n}\right).
\]
A vector $\alpha \in \R^n$ is called an {\em inverse eigenvector} of $M$ if 
\begin{equation}\label{eq:iev_def}
M \alpha = \alpha^{-1},
\end{equation}
equivalently, $\alpha$ is a real solution of the system of quadratic polynomial equations
\begin{equation}\label{eq:ievpoldef}
    x_k \sum_j m_{kj} x_j = 1, \quad k \in [n]
\end{equation}
where $[n]= \{ 1, \ldots, n\}$. 

We denote the set of inverse eigenvectors of $M$ by $\I(M)$.  Note that for $\alpha \in \I(M)$, we have 
\[
\sum_i \frac 1 {(M \alpha)_i^2} = \sum_i \alpha_i^2 = |\alpha|^2.
\]
Therefore, by using \eqref{eq:Malpha},   Theorem~\ref{thm:strongpol}  is implied by the next statement. 

\begin{theorem}\label{thm:smalliev}
       Every correlation matrix $M$ has an inverse eigenvector with Euclidean norm at most~$\sqrt{n}$.
\end{theorem} 

It suffices to prove Theorem~\ref{thm:smalliev} for positive definite correlation matrices. Indeed, let $M$ be a singular correlation matrix. For each $0<t<1$, define
\[
M_t=(1-t)M+tI_n,
\]
which is a positive definite correlation matrix. Assuming Theorem~\ref{thm:smalliev} holds for such matrices,  there exists $\alpha_t \in \I(M_t)$ with $|\alpha_t|^2 \leq n$. Choosing a convergent subsequence from these $\alpha_t$'s as $t \to 0$ yields a limit point $\alpha$ with $|\alpha|^2 \leq n$. Since $M_t\to M$ and $\alpha_t\to\alpha$, passing to the limit in the defining equations~\eqref{eq:ievpoldef} shows that $\alpha\in\I(M)$. Therefore, $M$ has an inverse eigenvector whose norm is at most $\sqrt{n}$.

\smallskip

Hence, from now on we assume that $M$ is a positive definite matrix with $1$'s along the diagonal. Then, $\E_M$ is an ellipsoid, and $\I(M)$ has a very natural geometric description, see Figure~\ref{fig1}. For $x = (x_1, \ldots, x_n) \in \R^n$, let $P(x):= x_1 \cdots x_n$.

\begin{proposition}[\cite{Ambrus2009,LeungLiRakesh2007,leung2008}]
\label{prop:IEV_Gram}
The inverse eigenvectors of $M$ are
precisely the maximizers of the function $|P(x)|$ on the intersections of the ellipsoid $\E_M$ with the orthants of $\R^n$. Consequently, $|\I(M)| = 2^n$.
\end{proposition}

\begin{proof}
The objective function $|P(x)|$ vanishes on the coordinate hyperplanes and is differentiable everywhere else. Thus, by compactness of $\E_M$, $|P|$ attains a maximum on $\E_M$ in every orthant.  Local maxima satisfy the Lagrange multiplier criterion. Since 
\[
\nabla P(x) = P(x) x^{-1} \quad \text{and} \quad \nabla x^\top M x = 2 M x,
\]
this implies that at any local maximum $\alpha$, we must have $\alpha^{-1} = \lambda M \alpha$, and the defining equation of $\E_M$ shows that $\lambda=1$. Therefore, any maximizer of  $|P(x)|$ on an orthant section of $\E_M$ is an inverse eigenvector.

Conversely, let $\alpha \in \I(M)$, and without loss of generality assume that $\alpha \in \R^n_+$.  Let $\HH_\alpha = \{ x \in \R^n_+: P(x) \geq P(\alpha) \}$ be the superlevel set of $P$, and $E$ be the solid ellipsoid bounded by $\E_M$. These are strictly convex sets. Equation \eqref{eq:iev_def} and the above gradient formulae along with the fact  $0 \in E$ show that the hyperplane $H_\alpha$ through $\alpha$ with normal vector $\alpha^{-1}$ is a separating hyperplane for $E$ and $\HH_\alpha$.  By strict convexity, $\E_M \cap \HH_\alpha = \{ \alpha \}$. Thus, $\alpha$ is the unique maximizer of $|P(x)|$ on $\E_M \cap \R^n_+$. Applying the same argument in every orthant completes the proof.
\end{proof}

\begin{figure}[h]
    \centering
    \includegraphics[width=0.5\linewidth]{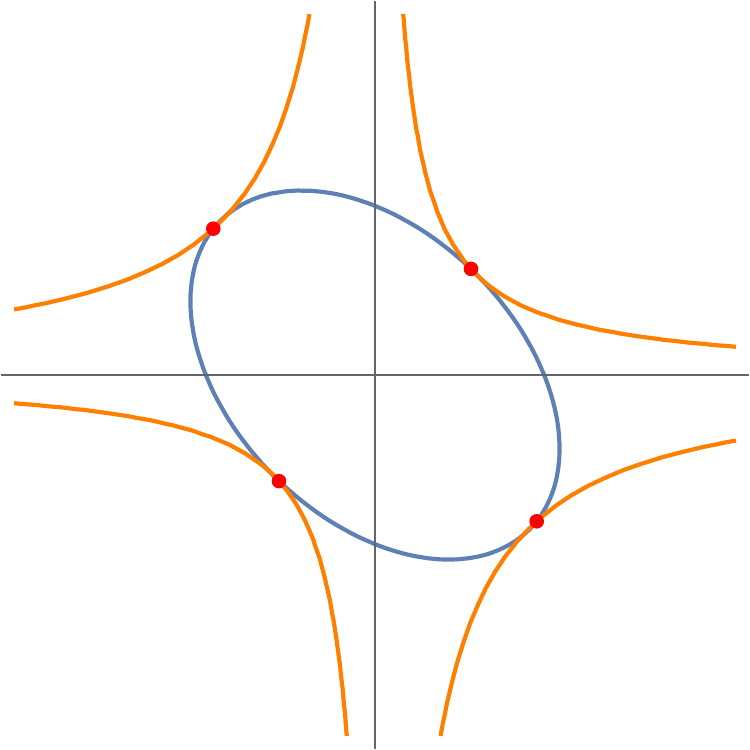}
    \caption{Inverse eigenvectors as points of tangency}
    \label{fig1}
\end{figure}

Next, we derive a determinant identity established by Ouimet and Greaves~\cite[Proof of Theorem 1]{OuimetGreaves2026}. Suppose that $\alpha \neq \beta \in \I(M)$. Then 
\[
M(\alpha - \beta) = \alpha^{-1} - \beta^{-1} =\D(\alpha^{-1})\D(\beta^{-1}) (\beta - \alpha) 
\]
where for a vector $x \in \R^n$, $\D(x)\in \R^{n \times n}$ is the diagonal matrix with diagonal $x$. 
Let 
\[
\gamma =\D(\beta^{-1}) (\alpha - \beta).
\]
Then $\gamma \neq 0$, and 
\[
\bigl(\D(\alpha) M\D(\beta) \bigr) \gamma =\D(\alpha) M ( \alpha - \beta) =\D(\alpha) \D(\alpha^{-1})\D(\beta^{-1}) (\beta - \alpha) = - \gamma.
\]
Equivalently, 
\[
\bigl(I_n +\D(\alpha) M\D(\beta) \bigr)\gamma = 0.
\]
This shows that for a fixed $\alpha$, the matrix $I_n +\D(\alpha) M\D(\beta) $ is singular for every $\beta \in \I(M) \setminus \{ \alpha \}$.

Accordingly, for each $\alpha \in \I(M)$, introduce
\begin{equation}\label{eq:padef}
p_\alpha (x) := \det \bigl(I_n +\D(\alpha) M\D(x) \bigr).
\end{equation}
Then $p_\alpha(\beta) = 0$ for every $\beta \in \I(M) \setminus \{ \alpha \}$. Furthermore, $p_\alpha(\alpha) >0$, since $\D(\alpha)M \D(\alpha)$ is positive definite, and hence so is
$I_n+\D(\alpha)M \D(\alpha)$. 
Define
\begin{equation}\label{eq:cdef}
c_\alpha = \frac{1}{ \det\bigl(I_n+\D(\alpha) M \D(\alpha)\bigr)} >0.
\end{equation}
Then for every $\alpha, \beta \in \I(M)$, 
\begin{equation}\label{eq:deltaab}
c_\alpha p_\alpha(\beta)  = \delta_{\alpha\beta}, 
\end{equation}
where $\delta_{\alpha\beta}$ denotes the Kronecker delta.

By expanding the determinant one readily obtains that $p_\alpha(x)$ is a polynomial in the  variables $x_1, \ldots, x_n$ that has degree at most $1$ in each variable. Such polynomials will be called  {\em multi-affine}. They form a vector space over the reals, that we denote by $\LL_n$. A basis of $\LL_n$ is given by the set of monomials of the form $x_1^{\delta_1} \cdots x_n^{\delta_n}$, where  $\delta_i \in \{ 0,1 \}$ for each $i = 1, \ldots, n$. Consequently, $\dim \mathcal{L}_n=2^n$.

Formula~\eqref{eq:deltaab} shows that the functions $\{ p_\alpha: \ \alpha \in \I(M) \}$ are linearly independent in~$\LL_n$. Since $|\I(M)| = \dim \mathcal{L}_n$, they form a basis in $\LL_n$. Therefore, we obtain the following Lagrange interpolation formula. 

\begin{theorem}[Interpolation formula for multi-affine polynomials]\label{thm:interpolation}
    Any polynomial $q \in \LL_n$ satisfies the representation
    \[
    q (x) = \sum_{\alpha \in \I(M)}  q(\alpha)  c_\alpha p_\alpha (x).
    \]
\end{theorem}

This interpolation formula is equivalent to the basis case of \cite[Theorem 2.3]{MartinezOrtegaMoreno2026_Aomoto}. It contains considerably more information than the estimate we need: it yields a decomposition of the identity analogous to those arising from the eigenvectors of a positive definite matrix or from contact points between the unit sphere and a convex body in John position~\cite{john1948extremum, Ball1997}. Note that the identities below are explicitly contained in~\cite[Equation (7)]{OuimetGreaves2026}, specialized to $R = M$, $\mathbf{y} = 0$, and $c_i = 1$, and are equivalent in coefficient coordinates to~\cite[Theorem 2.4]{MartinezOrtegaMoreno2026_Aomoto}.

\begin{theorem}[Inverse Spectral Theorem for positive definite correlation matrices] \label{thm:inversespectral}
Let $M \in \R^{n \times n}$ be a positive definite correlation matrix. Then 
\begin{equation}\label{eq:iev_decomposition}
\sum_{\alpha \in \I(M)} c_\alpha\,\alpha\otimes\alpha = I_n 
\end{equation}
and
\begin{equation}\label{eq:sumc}
\sum_{\alpha \in \mathcal{I}(M)} c_\alpha=1.
\end{equation}
\end{theorem}

\begin{proof}
Expanding the determinant in \eqref{eq:padef} and using $m_{ii}=1$ for every
$i\in[n]$, we readily obtain
\begin{equation}\label{eq:paformula}
p_\alpha(x)
=
1+\sum_i\alpha_i x_i+\widetilde p_\alpha(x),
\end{equation}
where $\widetilde p_\alpha$ consists of terms of total degree at least $2$.

We combine this expansion with the interpolation formula from
Theorem~\ref{thm:interpolation}. First, applying it to $q\equiv1$ and
evaluating at $x=0$ gives
\[
1
=
\sum_{\alpha\in\I(M)}c_\alpha p_\alpha(0)
=
\sum_{\alpha\in\I(M)}c_\alpha,
\]
which proves \eqref{eq:sumc}.

Next, fix $k\in[n]$ and use the interpolation formula for $q=x_k$:
\[
x_k
\equiv
\sum_{\alpha\in\I(M)}
c_\alpha\alpha_k p_\alpha(x).
\]
Comparing the linear terms in view of \eqref{eq:paformula}, we obtain
\[
\delta_{kl}
=
\sum_{\alpha\in\I(M)}
c_\alpha\alpha_k\alpha_l
\qquad (k,l\in[n]).
\]
These identities are equivalent to \eqref{eq:iev_decomposition}, completing the proof.
\end{proof}

The conclusion of the argument is now immediate.

\begin{proof}[Proof of Theorem~\ref{thm:smalliev}]
Comparing traces in \eqref{eq:iev_decomposition} yields that 
\[
\sum_{\alpha \in \I(M)} c_\alpha |\alpha|^2 = n 
\]
where the coefficients $c_\alpha$ are positive and sum to 1. Therefore, there exists $\alpha \in \I(M)$ that satisfies $|\alpha|^2 \leq n$.
\end{proof}
Together with \eqref{eq:Malpha}, this proves Theorem~\ref{thm:strongpol}; Theorem~\ref{thm:pol} then follows by the arithmetic--geometric mean inequality.

\bibliographystyle{siam}
\bibliography{EP_ref}

\bigskip

\footnotesize{Research was partially supported by the ERC Advanced Grant "GeoScape" no.  882971, by Hungarian National Research (NKFIH) grants no. KKP-133819, 147145, 147544, and 150151, which has been implemented with the support provided by the Ministry of Culture and Innovation of Hungary from the National Research, Development and Innovation Fund, financed under the ADVANCED-24 funding scheme. This research was funded by the grant 2024-1.2.8-TÉT-IPARI-CN-2025-00011,
with the support provided by the National Research,
Development and Innovation Office from the National Research,
Development and Innovation Fund, and financed under the
2024-1.2.8-TÉT-IPARI-CN funding scheme. }

\bigskip

\noindent
{\sc Gergely Ambrus}
\smallskip

\noindent
{\em Bolyai Institute, University of Szeged, Hungary,\\ and\\ 
 Alfréd Rényi Institute of Mathematics, Budapest, Hungary
 }
\smallskip

\noindent
e-mail address: \texttt{ambrus@server.math.u-szeged.hu; ambrus@renyi.hu}

\end{document}